\documentclass[12pt]{article}
\usepackage[latin1]{inputenc}
\usepackage[T1]{fontenc}
\usepackage{amsmath}
\usepackage{amsfonts}
\usepackage{amsthm}
\theoremstyle{plain}
\newtheorem{theorem}{Theorem}[section]
\newtheorem{coro}{Corollary}[section]
\newtheorem{lemma}{Lemma}[section]

\newtheorem{propo}{Proposition}[section]
\theoremstyle{definition}

\newtheorem{example}{Example}[section]
\newtheorem{remark}{Remark}[section]

\newcommand{\N}{\mathbb{N}}
\newcommand{\Z}{\mathbb{Z}}

\newcommand{\E}{\mathbb{E}}
\newcommand{\Var}{\operatorname{Var}}
\newcommand{\PP}{\mathbb{P}}

\newcommand{\C}{\mathbb{C}}

\newcommand{\id}{\operatorname{id}}

\newcommand{\K}{\mathbb{S}}

\title{Strictly stationary solutions of ARMA equations in Banach spaces}
\author{
Felix Spangenberg\thanks{Institut f\"ur Mathematische Stochastik, TU
Braunschweig,  Pockelsstra{\ss}e 14, D-38106 Braunschweig, Germany
\texttt{f.spangenberg@tu-bs.de}}}
\begin{document}
\maketitle

\begin{abstract}We obtain necessary and sufficient conditions for the existence of strictly stationary solutions of ARMA equations in Banach spaces with independent and identically distributed noise under certain assumptions. First, we obtain conditions for ARMA($1$,$q$) equations by excluding zero and the unit circle from the spectrum of the operator of the AR part. We then extend this to ARMA($p$,$q$) equations. Finally, we discuss various examples.
\end{abstract}

\section{Introduction}

An ARMA($p$,$q$) process is a stochastic process $(Y_t)_{t \in \Z}$ that fulfils the following recursion equation (ARMA, AutoRegressive Moving Average equation)
\[Y_t - a_1 Y_{t-1} - \ldots - a_p Y_{t -p}= b_0 Z_t + \ldots + b_q Z_{t-q}, \quad t \in \Z.\]
Here $a_1,\ldots,a_p$ and $b_0,\ldots, b_q$ are usually complex numbers and $(Z_t)_{t \in \Z}$ is a sequence of random variables, which are mostly either i.i.d. or uncorrelated. Such a sequence is called white noise. $Y$ and $Z$ are complex-valued stochastic processes as this facilitates many technical problems. An excellent introduction into time series in general and ARMA processes in particular is the monograph by Brockwell and Davis \cite{BrockwellDavis}.

A natural extension is to consider multivariate ARMA processes and one can take a further step by looking at ARMA processes in infinite dimensional vector spaces. For an introduction into time series in Banach spaces, see the monograph by Bosq \cite{Bosq} and the survey articles by Hörmann and Kokoszka \cite{Hoermann} and Mas and Pumo \cite{Mas}. ARMA processes in Banach spaces can be applied for example in climate prediction and financial modelling, see e.g. \cite{Besse} and \cite{Klueppelberg}. For work in functional data with heavy tails, see for instance the article by Meinguet and Segers \cite{Segers} which deals with extreme value theory of functional time series.

In this article we study conditions for the existence of solutions of ARMA equations in a separable complex Banach space $\mathcal{B}$ of the form
\begin{equation}Y_t - A_1 Y_{t-1} - \ldots - A_p Y_{t -p}= B_0 Z_t + \ldots + B_q Z_{t-q}, \quad t \in \Z, \label{ARMAEQ}\end{equation}
where $A_1,\ldots, A_p$ and $B_0,\ldots,B_q$ are linear continuous operators in $\mathcal{B}$. Any $\mathcal{B}$-valued stochastic process $(Y_t)_{t \in \Z}$ which satisfies this equation is called a solution of the ARMA($p$,$q$) equation or an ARMA($p$,$q$) process.
We investigate strictly stationary solutions of ARMA equations, where the white noise  $(Z_t)_{t \in \Z}$ is a series of Banach-space-valued i.i.d. random variables. This extends the work by Bosq \cite{Bosq} who deals with causal weakly and strictly stationary solutions of AR equations with finite second-moment white noise.

Firstly, we generalise results by Brockwell and Lindner \cite{BrockwellLindner} and Brockwell, Lindner and Vollenbr\"ocker \cite{BrockLindVoll}. They give necessary and sufficient conditions for the existence of strictly stationary solutions for multivariate ARMA($1$,$q$) equations in terms of the eigenvalues of the matrix of the AR part and $\log^+$-moment conditions on the white noise.  Our approach is to partially generalise this by investigating the spectrum of the operator of the AR($1$) part. Assuming that the spectrum of the autoregressive operator does not contain zero and has empty intersection with the unit circle, we derive necessary and sufficient conditions for the existence of a strictly stationary solution in terms of finiteness of $\log^+$-moments. Secondly, we extend our results to  ARMA($p$,$q$) processes by using their representations as  $\mathcal{B}^p$-valued ARMA($1$,$q$) processes. Thirdly, we give an additional representation of the solution as a moving average process of infinite order by employing Laurent series. 
Finally, we look at what can happen in the case when zero is in the spectrum or the intersection of the spectrum and the unit circle is nonempty by looking at various examples.

\section{Results}

We now give our main results. We first give necessary and sufficient conditions for the existence of strictly stationary ARMA($1$,$q$) processes when the spectrum does not contain zero and elements of the unit circle.  We extend these results to ARMA($p$,$q$) processes. Finally, we give a representation of the solution as a moving average process of infinite order using Laurent series.

The proof of the characterisation of the existence of strictly stationary  multivariate ARMA($1$,$q$) processes in \cite{BrockLindVoll} makes use of the Jordan canonical decomposition. A Jordan decomposition for operators in Banach spaces does not exist in full generality. Therefore, we restrict ourselves to operators in the AR part whose spectra, denoted by $\sigma(A_1)$, do not contain elements of the unit circle $\K$. Let $A_1$ be an operator with $\sigma(A_1)\cap \K = \emptyset$. We can diagonalise $A_1$ via an invertible linear continuous operator $S: \mathcal{B}_1 \oplus \mathcal{B}_2 \to \mathcal{B}$ such that
\[ S^{-1}A_1S= \begin{pmatrix} \Lambda_1 & 0 \\ 0 & \Lambda_2 \end{pmatrix},\]
where $\Lambda_1: \mathcal{B}_1 \to \mathcal{B}_1$ and $\Lambda_2: \mathcal{B}_2 \to \mathcal{B}_2$ are bounded linear operators with $\sigma(\Lambda_1) = \{z \in \C: |z|<1\} \cap \sigma(A_1)$ and $\sigma(\Lambda_2) = \{z \in \C: |z|>1\}\cap \sigma(A_1)$. This can be proven by using functional calculus for holomorphic functions. For a proof see Theorem 6.17 in Kato \cite{Kato} Chapter III.6.4: $S^{-1}$ is given by $S^{-1}=(P,\id-P)$, where $P = - \frac{1}{2 \pi i} \int_{\gamma} R(z,A_1) \, dz$, with $\gamma$ being a path following the unit circle and $R(z,A_1)$ being the resolvent operator of $A_1$. $\mathcal{B}_1$ and $\mathcal{B}_2$ are the images of $P$ and $\id -P$. The space $\mathcal{B}_1 \oplus \mathcal{B}_2:=\mathcal{B}_1 \times \mathcal{B}_2$ is endowed with a norm, e.g. the maximum norm, that renders it into a Banach space, and $S$ is an isomorphism between $\mathcal{B}_1 \oplus \mathcal{B}_2$ and $\mathcal{B}$. Hence $\mathcal{B}_1 \oplus \mathcal{B}_2$ and $\mathcal{B}$  can be identified.

An operator $A_1$ that fulfils $\sigma(A_1)\cap \K = \emptyset$ is called \textsl{hyperbolic}. The notion of hyperbolicity is used in the theory of operator semigroups in the context of stability, see for example \cite{Engel}.

We use the same diagonalisation of the ARMA($1$,$q$) equation as  in \cite{BrockLindVoll}:
the ARMA($1$,$q$) equation
\begin{equation}Y_t - A_1 Y_{t-1}= B_0 Z_t + \ldots + B_q Z_{t-q}, \quad t \in \Z \label{ARMA1q}\end{equation}
has a strictly stationary solution $(Y_t)_{t \in \Z}$ if and only if the corresponding equation for $X_t:=S^{-1} Y_t$
\[X_t - S^{-1} A_1 S X_{t-1}= S^{-1} B_0 Z_{t} + \cdots + S^{-1} B_q Z_{t-q}, \quad t \in \Z,\]
has a strictly stationary solution. We define $I_i$ as the projection on the $i$-th component of $\mathcal{B}_1 \oplus \mathcal{B}_2$ and $X_t^{(i)}=I_i X_t$ for $i=1,2$. Then there is a strictly stationary solution of the original equation if and only if there are strictly stationary solutions for
\begin{equation}X_t^{(i)} - \Lambda_i X_{t-1}^{(i)}= I_i S^{-1} B_0 Z_{t} + \cdots + I_i S^{-1} B_q Z_{t-q}, \quad t \in \Z, \quad i=1,2. \label{ARMAcomponentwise}\end{equation}

Recall that the spectral radius $r(A_1)=\lim_{n \to \infty} \sqrt[n]{||A_1^n||}$ coincides with $\sup \{|\lambda|: \lambda \in \sigma(A_1)\}$ which justifies the name spectral radius. It is easy to see that $r(A_1)<1$ if and only if $A_1$ is uniformly exponentially stable, i.e. there exist constants $a \geq0$, $0<b<1$ such that $||A_1^j||\leq a b^{j}$ for all $j \in \N$. The latter condition corresponds to condition $(c_1)$ on p. 74 in \cite{Bosq}.
Hence the condition $\sigma(A) \subset \{z \in \C: |z|<1\}$ is equivalent to the conditions in Lemma 3.1 of \cite{Bosq}.

We now give an extension of Theorem 1 in \cite{BrockLindVoll}:

\begin{theorem}\label{necessity}
Let $\mathcal{B}$ be a complex separable Banach space and let $(Z_t)_{t \in \Z}$ be an i.i.d. sequence of $\mathcal{B}$-valued random variables. Let $A_1$ and $B_0,\ldots,B_q$ be linear continuous operators in $\mathcal{B}$ and assume $\sigma(A_1) \cap \K = \emptyset$. Let $S$ be the operator as given above such that $S^{-1} A_1 S=\begin{pmatrix} \Lambda_1 & 0 \\ 0 & \Lambda_2 \end{pmatrix}$ . Then the ARMA($1$,$q$) equation (\ref{ARMA1q}) has a strictly stationary solution $Y=(Y_t)_{t \in \Z}$ if
\[\E \log^+ \Big |\Big| \Big( \sum_{k=0}^q A_1^{q-k} B_k \Big)Z_0\Big|\Big| < \infty.\]
This moment condition is also necessary under the additional assumption $0 \notin \sigma(A_1)$.
The strictly stationary solution is unique in both cases.
It is given by $Y_t=S X_t$ with
\[X_t^{(1)} = \sum_{j=0}^{q-1} (\sum_{k=0}^{j} \Lambda_1^{j-k} I_1 S^{-1} B_k)  Z_{t-j} + \sum_{j=q}^{\infty} \Lambda_1^{j-q}(\sum_{k=0}^{q} \Lambda_1^{q-k} I_1 S^{-1} B_k)  Z_{t-j}\]
and
\[X_t^{(2)} = - \sum_{j= 1-q}^{\infty} \Lambda_2^{-j-q} (\sum_{k=(1-j) \vee 0}^q \Lambda_2^{q-k} I_2 S^{-1}B_k)Z_{t+j},\]
where the series defining $X_t^{(1)}$ and $X_t^{(2)}$  converge almost surely absolutely.

\end{theorem}

\begin{proof}
The proof follows along the lines of the proofs of Theorem 1 and Corollary 1 in \cite{BrockLindVoll}. We first consider the case $\sigma(A_1) \subset \{z \in \C: |z|<1\}$, then the case $\sigma(A_1) \subset \{z \in \C: 1<|z|\}$ and finally the general case.

Case 1: $\sigma(A_1) \subset \{z \in \C: |z|<1\}$
i) Sufficiency: Assume that the moment condition is fulfilled. The sufficiency of the conditions follows along the lines of Section 3.2 in \cite{BrockLindVoll}. Note that because $r(A_1)<1$, there are $a>0$ and $0<b<1$ such that $||A_1^j||\leq a b^j$. We show that this moment condition is sufficient for the almost sure absolute convergence of
\[Y_t = \sum_{j=0}^{q-1} \Big( \sum_{k=0}^{j} A_1^{j-k} B_k\Big) Z_{t-j} + \sum_{j=q}^{\infty} A_1^{j-q} \Big( \sum_{k=0}^{q} A_1^{q-k} B_k\Big) Z_{t-j}.\]
By using the Borel-Cantelli-Lemma, we show that only finitely many summands have norm greater than $b'^j$ for $b<b'<1$:
\begin{eqnarray*}
&& \sum_{j=q}^{\infty} \PP\Big[\Big|\Big|A_1^{j-q}\Big(\sum_{k=0}^q A_1^{q-k} B_k \Big) Z_{t-j}\Big|\Big|>b'^j\Big]\\
&\leq& \sum_{j=q}^{\infty} \PP\Big[ab^{-q}\Big|\Big|\Big(\sum_{k=0}^q A_1^{q-k} B_k \Big) Z_{t-j}\Big|\Big|>(\frac{b'}{b})^j\Big]\\
&=& \sum_{j=q}^{\infty} \PP\Big[\log^+\Big( ab^{-q}\Big|\Big|\Big(\sum_{k=0}^q A_1^{q-k} B_k \Big) Z_{t-j}\Big|\Big|\Big)>j \log \frac{b'}{b}\Big].
\end{eqnarray*}
The last series is finite because of our moment assumption. Hence the series defining $Y_t$ converges almost surely absolutely. Obviously,
$(Y_t)_{t \in \Z}$ is a strictly stationary process and one can check that it defines a solution of (\ref{ARMA1q}).

ii) Necessity and uniqueness: Assume that there is a strictly stationary solution. By iterating the ARMA($1$,$q$) equation (\ref{ARMA1q}) (see equation (22) in \cite{BrockLindVoll}), we have
\begin{eqnarray}Y_t&=& \sum_{j=0}^{q-1} \Big(\sum_{k=0}^j A_1^{j-k} B_k\Big) Z_{t-j} + \sum_{j=q}^{n-1} A_1^{j-q}\Big(\sum_{k=0}^q A_1^{q-k} B_k\Big) Z_{t-j} \nonumber\\
&&+ \sum_{j=0}^{q-1} A_1^{n+j-q} \Big(\sum_{k=j+1}^q A_1^{q-k} B_k \Big) Z_{t-(n+j)} + A_1^{n} Y_{t-n}.\label{iteratedARMA}\end{eqnarray}

By taking the limit in probability as $n \to \infty$, the last two summands converge to $0$ in probability (as they converge to $0$ in distribution), since $(Y_t)_{t \in \Z}$ and $(Z_t)_{t \in \Z}$ are both strictly stationary. Hence we get 
\[Y_t =  \sum_{j=0}^{q-1} \Big(\sum_{k=0}^j A_1^{j-k} B_k\Big) Z_{t-j} + \PP-\lim_{n \to \infty} \sum_{j=q}^{n-1} A_1^{j-q}\Big(\sum_{k=0}^q A_1^{q-k} B_k\Big) Z_{t-j}.\] This shows uniqueness. Now assume $0 \notin \sigma(A_1)$.
The Itô-Nisio-Theorem is also valid in Banach spaces (see Theorem 6.1 in \cite{Ledoux}), therefore we also get almost sure convergence. Hence only finitely many summands may have norm greater than 1, which gives the last inequality of the following (in)equalities by the Borel-Cantelli-Lemma:
\begin{eqnarray*}
&&\sum_{j=q}^{\infty} \PP \Big[\Big|\Big| \Big(\sum_{k=0}^q A_1^{q-k} B_k\Big) Z_{0} \Big|\Big| >{\underbrace{||A_1^{-1}||}_{>1}}^{j-q} \Big] \\
&\leq& \sum_{j=q}^{\infty} \PP \Big[||A_1^{q-j}||\Big|\Big| A_1^{j-q}\Big(\sum_{k=0}^q A_1^{-k} B_k\Big) Z_{-j} \Big|\Big| >||A_1^{q-j}|| \Big] \\
&=&\sum_{j=q}^{\infty} \PP \Big[\Big|\Big| A_1^{j-q}\Big(\sum_{k=0}^q A_1^{q-k} B_k\Big) Z_{-j} \Big|\Big| >1 \Big] < \infty.
\end{eqnarray*}
This gives $\E \log^+ \Big |\Big| \Big( \sum_{k=0}^q A_1^{q-k} B_k \Big)Z_0\Big|\Big| < \infty$.

Case 2: $\sigma(A_1) \subset \{z \in \C: 1<|z|\}$
i) Sufficiency: Define
\begin{eqnarray*}
Y_u&:=& -\sum_{j=0}^{q-1} A_1^j \Big(\sum_{k=j+1}^q A_1^{-k} B_k\Big) Z_{u-j} - \sum_{j=-\infty}^{-1} A_1^{j}\Big(\sum_{k=0}^q A_1^{-k} B_k\Big) Z_{u-j}\\
&=&- \sum_{j=-\infty}^{q-1} A_1^{j}\Big(\sum_{k=\max(0,j+1)}^q A_1^{-k} B_k\Big) Z_{u-j}.
\end{eqnarray*}
As in case 1, it follow  from the moment condition that the defining series $Y_u$ converges almost surely absolutely  and by similar calculations it follows that $(Y_u)_{u \in \Z}$ is a strictly stationary solution of (\ref{ARMA1q}).

ii) Necessity and uniqueness: For obtaining the same result in the case $\sigma(A_1) \subset \{z \in \C: 1<|z|\}$, one proceeds as in subsection 3.1.2. in \cite{BrockLindVoll}: Note that $\sigma(A_1^{-1})\subset \{z \in \C: 0<|z|<1\}$ by holomorphic functional calculus. One obtains (see equation (25) in \cite{BrockLindVoll})
\begin{eqnarray*}Y_u&=& -\sum_{j=0}^{q-1} A_1^j \Big(\sum_{k=j+1}^q A_1^{-k} B_k\Big) Z_{u-j} - \sum_{j=1}^{n-q} A_1^{-j}\Big(\sum_{k=0}^q A_1^{-k} B_k\Big) Z_{u+j}\\
&&- \sum_{j=0}^{q-1} A_1^{-n+j} \Big(\sum_{k=0}^j A_1^{-k} B_k \Big) Z_{u+n-j} + A_1^{-n} Y_{u+n}.\end{eqnarray*}
Again, the last two summands converge to $0$ in probability as $n\to \infty$ and \[\sum_{j=1}^{\infty} A_1^{-j}\Big(\sum_{k=0}^q A_1^{-k} B_k\Big) Z_{u+j}\] has to converge almost surely. This shows uniqueness and gives the same necessary moment condition by the same argument.

The general case: By the discussion preceding Theorem \ref{necessity}, there exists a strictly stationary solution if and only if (\ref{ARMAcomponentwise}) admits a strictly stationary solution for $i=1,2$. By the cases 1 and 2, this holds if (and only if in the case $0 \notin \sigma(A_1)$)
\[\E \log^+ ||(\sum_{k=0}^q \Lambda_i^{q-k} I_i S^{-1} B_k)Z_0|| < \infty \quad \mbox{for }i=1,2\]  which is in turn equivalent to
\begin{eqnarray*}
\infty&>&\E \log^+ ||(\sum_{k=0}^q \begin{pmatrix} \Lambda_1 & 0 \\ 0 & \Lambda_2 \end{pmatrix}^{q-k} S^{-1} B_k)Z_0||\\
&=&\E \log^+ ||(\sum_{k=0}^q (S^{-1}A_1S)^{q-k} S^{-1} B_k)Z_0||\\
&=&\E \log^+ ||(\sum_{k=0}^q S^{-1} A_1^{q-k}  B_k)Z_0||.
\end{eqnarray*}
Finally, this is true if and only if $\E \log^+ ||(\sum_{k=0}^q A_1^{q-k}  B_k)Z_0||<\infty$ because $S$ is invertible and continuous.
The uniqueness and the specific form of the solutions are clear from the cases 1 and 2 and the discussing preceding Theorem \ref{necessity}.
\end{proof}

One can hope that one can also use a continuous operator $A_1$ with the property that $\lim_{n \to \infty} A_1^n x=0$ for all $x \in \mathcal{B}$ as we then still have $A_1^n Z \stackrel{d}{\rightarrow}0$ for a random variable $Z$ and hence for $A_1^n Z_{t-n}$ for a strictly stationary process $(Z_t)_{t \in \Z}$. An operator with this property is called \textit{strongly stable}, see \cite{Eisner}, Definition 2.1. There are strongly stable operators with spectral radius one. For example, consider a multiplication operator $A_1$ on $\ell^2(\N)$ defined by $A_1(x_0,x_1, \ldots) = (\lambda_0 x_0, \lambda_1 x_1, \ldots)$. Recall that $\sigma(A_1)=\overline{\{\lambda_i,i \in \N\}}$. If we choose a sequence $0<\lambda_i <1$ tending to 1, then $r(A_1)=1$, but $A_1$ is strongly stable. However, if $r(A_1)=1$, then $||A_1^n||$ does not converge exponentially fast to zero and hence sufficient $\log^+$-moment conditions cannot be derived in this way. Still, we have the following:

\begin{coro} Let $\mathcal{B}$ be a complex separable Banach space and let $(Z_t)_{t \in \Z}$ be an i.i.d. sequence of $\mathcal{B}$-valued random variables. Let $A_1$ and $B_0,\ldots,B_q$ be linear continuous operators in $\mathcal{B}$. Further assume that $A_1$ is strongly stable. Then the ARMA($1$,$q$) equation (\ref{ARMA1q}) has a strictly stationary solution $Y=(Y_t)_{t \in \Z}$ if and only if
\[\PP-\lim_{n \to \infty} \sum_{j=q}^{n-1} A_1^{j-q}\Big(\sum_{k=0}^q A_1^{q-k} B_k\Big) Z_{t-j} \mbox{ exists.}\]
If there is a strictly stationary solution, it is unique and is given by
\[Y_t =  \sum_{j=0}^{q-1} \Big(\sum_{k=0}^j A_1^{j-k} B_k\Big) Z_{t-j} + \PP-\lim_{n \to \infty} \sum_{j=q}^{n-1} A_1^{j-q}\Big(\sum_{k=0}^q A_1^{q-k} B_k\Big) Z_{t-j}.\]
\end{coro}

\begin{proof}
The necessity and uniqueness follows as in the proof of Theorem \ref{necessity}-case 1 since the strong stability implies that $\sum_{j=0}^{q-1} A_1^{n+j-q} \Big(\sum_{k=j+1}^q A_1^{q-k} B_k \Big) Z_{t-(n+j)}\stackrel{\PP}{\rightarrow}0$ and  $A_1^{n} Y_{t-n}\stackrel{\PP}{\rightarrow}0$ as $n \to \infty$ as observed above, so that (\ref{iteratedARMA}) gives the desired convergence conditions. The sufficiency follows from the same arguments as given in the proof of Theorem \ref{necessity} since convergence in probability is preserved by application of continuous operators.
\end{proof}
An example of a strictly stationary AR($1$) process with strongly stable $A_1$ but with $r(A_1)=1$ is given in Proposition \ref{stronglystablear}.

We can extend our results to ARMA($p$,$q$) processes by using the representation as an ARMA($1$,$q$) process with new state space $\mathcal{B}^p$ endowed with a suitable norm, e.g. the maximum of the norms of the components. We follow here Section 5.1 of Bosq \cite{Bosq}. If we have the ARMA($p$,$q$) equation $Y_t - A_1 Y_{t-1} - \ldots - A_p Y_{t-p} = B_0 Z_t+B_1 Z_{t-1}+\ldots+B_q Z_{t-q}$, we set
\begin{eqnarray}
\widetilde{Y}_t&=&(Y_t,\ldots, Y_{t-p+1})^T,\nonumber\\ 
\widetilde{B_k}&=&(B_k,0,\ldots,0)^T,\label{newOperator1}\\
\widetilde{Z_t}&=&(Z_t,0,\ldots,0)^T. \nonumber
\end{eqnarray}
We further define the operator $A$ on $\mathcal{B}^p$ by
\begin{equation}A= \begin{pmatrix}A_1 & A_2 & \cdots & \cdots & A_p \\ \id & 0 & \cdots & \cdots & 0\\ 0 & \id & \cdots & \cdots & 0 \\ \vdots & \vdots & \ddots & \vdots & \vdots \\ 0 & \cdots & \cdots & \id & 0 \end{pmatrix}.\label{newOperator2}\end{equation}
We can now regard the old ARMA($p$,$q$) equation as the new ARMA($1$,$q$) equation
\begin{equation}\widetilde{Y_t}-A \widetilde{Y_{t-1}} = \widetilde{B_0} \widetilde{Z_t}+\widetilde{B_1} \widetilde{Z_{t-1}}+\ldots+\widetilde{B_q} \widetilde{Z_{t-q}}.\label{newARMA}\end{equation}
Note that $(\widetilde{Z_t})_{t \in \Z}$ is also strict white noise.
We now consider the operator-valued polynomial $Q$ defined by
\begin{equation}Q(z)= z^p \id - z^{p-1} A_1 - \cdots -z A_{p-1} - A_{p}.\label{operatorvaluedpolynomial}\end{equation}

\begin{lemma} We have \[\sigma(A):= \{z \in \C| z \id - A \mbox{ is not invertible}\} = \{z \in \C|Q(z) \mbox{ is not invertible}\}.\] \end{lemma}
\begin{proof} This follows from the proof of Theorem 5.2 p. 130  in Bosq \cite{Bosq}. Bosq states only the inclusion $\subset$ in his proof but his arguments also give equality. \end{proof}

We now apply this representation to Theorem \ref{necessity}:

\begin{theorem}
Let $A_1,\cdots,A_p$ and $B_0,\cdots,B_q$ be continuous linear operators in a separable complex Banach space $\mathcal{B}$. Let $(Z_t)_{t \in \Z}$ be an i.i.d. sequence of $\mathcal{B}$-valued random variables. Define $Q$ as in (\ref{operatorvaluedpolynomial}) and assume that $Q(z)$ is invertible for all $z \in \C$ with $|z|=1$. Define $\widetilde{B_k}$, $\widetilde{Z_t}$ and $A$ as in (\ref{newOperator1}) and (\ref{newOperator2}). Then a strictly stationary solution of the ARMA($p$,$q$) equation (\ref{ARMAEQ}) exists if
\begin{equation}\E \log^+ \Big |\Big| \Big( \sum_{k=0}^q A^{q-k} \widetilde{B_k} \Big)\widetilde{Z_0}\Big|\Big| < \infty.\label{newmoment}\end{equation}
If $Q(0)$ is also invertible, then this moment condition is also necessary.
The stationary solution is unique.
\end{theorem} 

\begin{proof}
There is a solution for the ARMA($p$,$q$) equation (\ref{ARMAEQ}) if and only if there is a solution for the rewritten ARMA($1$,$q$) equation (\ref{newARMA}). The spectrum  $\sigma(A)$ of $A$ is $\{z \in \C|Q(z) \mbox{ is not invertible}\}$. Hence by Theorem \ref{necessity}, a strictly stationary solution exists if (and only if in the case $Q(0)$ is invertible) (\ref{newmoment}) holds. If there is a solution for the rewritten ARMA($1$,$q$) equation, it is unique and hence so is the corresponding solution of the ARMA($p$,$q$) equation that is given by its first component.\end{proof}

The theory of holomorphic functions extends to holomorphic operator-valued functions, see e.g. \cite{Conway}, VII §4. Especially, if we have an operator-valued function $f$ that is holomorphic on an annulus around the unit circle, then it has a representation as a Laurent series $\sum_{n=-\infty}^{\infty} z^n \phi_n$. The operators $\phi_n$ are given by $\phi_n = \frac{1}{2\pi} \int_{\gamma} z^{-n-1} f(z) \, dz$, where $\gamma$ is a path following the unit circle, see \cite{Conway}, VII Lemma 6.11. Actually, VII Lemma 6.11 in \cite{Conway} is formulated only for a punctured disc, but the proof of the Laurent Series development 1.11 in \cite{Conway2} only makes use of Cauchy's Theorem and Cauchy's Integral Formula. These are also true for operator-valued holomorphic functions, see \cite{Conway}, §4.1 and §4.2.

We now give another representation of the unique strictly stationary solution:

\begin{theorem}\label{sufficieny}
Let $(Z_t)_{t \in \Z}$ be strict white noise in a separable complex Banach space $\mathcal{B}$ satisfying $\E \log^+ ||Z_0|| < \infty$. Let $A_1,\cdots,A_p$ and $B_0,\cdots,B_q$ be continuous linear operators and let $Q$ be defined by (\ref{operatorvaluedpolynomial}). Assume $Q(z)$ is invertible for all $z \in \C$ with $|z|=1$. Then the unique strictly stationary solution $Y=(Y_t)_{t \in \Z}$ of the ARMA($p$,$q$) equation (\ref{ARMAEQ}) is given by
\[ Y_t = \sum_{k=-\infty}^{\infty} \psi_k Z_{t-k},\]
where the series converges almost surely absolutely and where $(\psi_k)_{k \in \Z}$ denote the coefficients of the Laurent series of $(\id-zA_1-z^2 A_2-\cdots z^p A_p)^{-1}(B_0+zB_1 + \ldots + z^q B_q)$ for an annulus containing the unit circle $\K$. The coefficients $(\psi_k)_{k \in \Z}$ are given by $\psi_n = \frac{1}{2\pi} \int_{\gamma} z^{-n-1} (\id-zA_1-z^2 A_2-\cdots z^p A_p)^{-1}(B_0+zB_1 + \ldots + z^q B_q)  \, dz$.
\end{theorem}

\begin{proof}
We first show that $Q(z)^{-1}$ is a holomorphic function. For that, observe that from the proof of Theorem 5.2 in \cite{Bosq} it holds
\[M(z)(z \id - A)N(z)=\begin{pmatrix} \id & 0 & \ldots & 0 & 0\\ 0 & \id & \ldots & \vdots & \vdots\\ \vdots & & & \id &0\\ 0 & \ldots & \ldots & 0 & Q(z)\end{pmatrix}=:R(z).\]
Here $N(z)$ is a suitable upper triangular matrix with diagonal entries $\id$ and upper entries of the form $z^k \id$. The matrix $M$ is given by
\[M(z)=\begin{pmatrix} 0 & -\id & 0 & \ldots & 0\\ 0 & 0 & -\id & \ldots & 0\\ \vdots & & &  &\\ 0 & \ldots & \ldots & 0 & -\id\\ \id&Q_1(z)&\ldots&\ldots& Q_{p-1}(z)\end{pmatrix},\]
where $Q_1(z),\ldots,Q_{p-1}(z)$ are operator-valued polynomials.
The inverse of $M$ is given by
\[M(z)^{-1}=\begin{pmatrix}  Q_1(z)& \ldots & \ldots& Q_{p-1}(z) & \id\\ -\id & 0 & 0 & \ldots & 0\\ 0 & \ddots & &  &\\ 0 & \ldots & -\id & 0 & 0\\ 0&\ldots&\ldots &-\id&0\end{pmatrix}.\]
Thus $M(z)^{-1}$ is a holomorphic function. The inverse of $N(z)$ can be easily given by Cramer's rule and is hence holomorphic as well. The inverse of $(z \id - A)$ is the resolvent of $A$ and it is known that the resolvent is a holomorphic function. Hence $R(z)^{-1}$ is holomorphic and hence so is $Q(z)^{-1}$.

Now observe that $(\id-zA_1-z^2 A_2-\cdots z^p A_p)^{-1}=(z^pQ(z^{-1}))^{-1}$. Hence this function is holomorphic on an annulus containing the unit circle. This shows that we can develop  $(\id-zA_1-z^2 A_2-\cdots z^p A_p)^{-1}(B_0+zB_1 + \ldots + z^q B_q)$ as a Laurent series by the discussion preceding this theorem.

Finally, we show that the series defining $Y_t$ converges almost surely absolutely:
the Laurent series $\sum_{k=-\infty}^{\infty} z^k \psi_k$ converges absolutely on an annulus containing the unit circle.  As in the proof of Theorem 3.1.1 in \cite{BrockwellDavis} we find the same exponential decay of $||\psi_k||$ as in the one-dimensional case: the Laurent series is absolutely convergent, hence $\sum_{n=1}^{\infty} (1+\varepsilon)^n \psi_n$ and $\sum_{n=1}^{\infty} (1-\varepsilon)^{-n} \psi_{-n}$ are also absolutely convergent for $\varepsilon>0$ small enough. Hence there are $a>0$ and $0<b<1$ such that $||\psi_n||<ab^{|n|}$.
The proof of the almost sure absolute convergence is the same as in the proof of Theorem \ref{necessity} and one can check that $(Y_t)_{t \in \Z}$ is a strictly stationary solution.

\end{proof}

\begin{remark} The assumption that $B_0,\ldots,B_q$ are continuous operators is in fact not needed. In Theorem \ref{necessity} the white noise could be an $E$-valued sequence, where $E$ is a measure space with $B_k:E \to \mathcal{B}$ being measurable mappings. In Theorem \ref{sufficieny}, the mappings can be continuous operators from a normed space $\mathcal{A}$ to $\mathcal{B}$ or also only measurable mappings, if we assume $\E[\log^+||B_0 Z_t +B_1 Z_{t-1} + \ldots +  B_q Z_{t-q}||]<\infty$. \end{remark}

If $(Z_t)_{t \in \Z}$ is weak white noise, i.e. a sequence of second order random variables with constant expectation and constant covariance operator and whose cross covariance operators vanish, then $Y_t = \sum_{k=-\infty}^{\infty} \psi_k Z_{t-k}$ is well-defined and defines a weakly stationary sequence and hence a weakly stationary solution of the ARMA equation. See Definition 2.4 and Definition 3.1 in \cite{Bosq} for the definitions of weak white noise and weak stationarity.
Hence we get a corollary which generalises Theorem 3.1 in \cite{Bosq}:
\begin{coro}
Let $A_1,\cdots,A_p$ and $B_0,\cdots,B_q$ be continuous linear operators in a separable complex Banach space $\mathcal{B}$. Define $Q$ as in (\ref{operatorvaluedpolynomial}) and assume that $Q(z)$ is invertible for all $z \in \C$ with $|z|=1$. Let $(Z_t)_{t \in \Z}$ be weak white noise with $\E[Z_0]=0$. The ARMA($p$,$q$) equation (\ref{ARMAEQ}) then has a unique weakly stationary solution $Y=(Y_t)_{t \in \Z}$ given by
\[ Y_t = \sum_{k=-\infty}^{\infty} \psi_k Z_{t-k},\]
where the series converges almost surely absolutely and in the $L^2$-sense.
\end{coro}
\begin{proof} The existence as an almost surely absolutely convergent series can be established similar to Proposition 3.1.1. in \cite{BrockwellDavis}. The sequence $(\sum_{k=-n}^{n} \psi_k Z_{t-k})_{n \in \N}$ is a Cauchy sequence in the space of square integrable random variables as
\[\E||\sum_{k=-n}^{n} \psi_k Z_{t-k}-\sum_{k=-m}^{m} \psi_k Z_{t-k}||^2\leq \Big(\sum_{k=-n}^{n} ||\psi_k|| -\sum_{k=-m}^{m} ||\psi_k||\Big)^2 \E||Z_0||^2,\]
which follows from the proof of Theorem 6.1 in \cite{Bosq}. The uniqueness follows along the lines of the proof of Lemma 3.1 in \cite{BrockwellLindner2}. However, the last argument of the proof uses Slutsky's Lemma. This cannot be applied, but the convergence in probability follows from Tchebychev's inequality. Finally, it it obvious that $(Y_t)_{t \in \Z}$ is weakly stationary and one can check that it is a solution.\end{proof}

\section{Examples}

\subsection{Weaker moment conditions if $\sigma(A_1) = \{0\}$}

We already know that a $\log^+$-moment is sufficient for the existence of a solution of the AR($1$) equation $Y_t - A_1Y_{t-1}=Z_t$ if $\sigma(A_1) = \{0\}$. We now give examples to show that this sufficient condition is not necessary in this case.

\begin{example} The first example is that $A_1$ is the zero operator or more generally nilpotent, i.e. there is a power $A_1^n$ that vanishes. The spectral radius is then $r(A_1)=0$, hence $\sigma(A_1)=\{0\}$. The series $\sum_{n=0}^{\infty} A_1^n Z_{t-n}$ always converges. Hence there is no necessary moment condition. \end{example} 

In the next example, the spectral radius of the operator $A_1$ vanishes, i.e. $r(A_1)=0$ and $\sigma(A_1)=\{0\}$ but $A_1$ is not nilpotent. Recall that such operators are called \textsl{quasinilpotent}.

\begin{example}
Consider the weighted right shift $A_1$ on $c_0(\N)$ given by $A_1(x_0,x_1,x_2,\ldots)=(0,a_1 x_0, a_2 x_1, \ldots)$ for a monotone sequence $a_i\geq0$ tending to zero. We then have $||A_1^n||=a_1 a_2 \cdots a_n$. We set $a_n=\frac{e^{-e^n}}{e^{-e^{n-1}}}$ for $n>1$ and $a_1= e^{-e}$, hence $||A_1^n||=e^{-e^n}$. We claim that the condition $\E[\log^+ \log^+ ||Z_0||]<\infty$ is sufficient for the existence of a solution of the corresponding AR($1$) equation and that this  moment condition is sharp inasmuch as there cannot be a better sufficient moment condition. What is more, the condition is indeed different from the usual $\log^+$-moment condition in that there are distributions with $\log^+\log^+$-moments without $\log^+$-moments.
\end{example}
\begin{proof}
Sufficiency:
We know from the proof of Theorem \ref{necessity} that there is a strictly stationary solution for $Y_t-A_1Y_{t-1}=Z_t$ if  $\sum_{n=0}^{\infty} A_1^n Z_{t-n}$ converges almost surely absolutely. Hence it suffices to show that $||A_1^n Z_{t-n}||>e^{-n}$ only finitely often. We show this by the Borel-Cantelli-Lemma. First note that there is a $K$ such that $\log(e^n-n)=n + \log(1-ne^{-n}) > n - \frac{1}{2}$ for $n \geq K$. Hence
\begin{eqnarray*}
&& \sum_{n=K}^{\infty} \PP[||A_1^n Z_{t-n}||> e^{-n}]\\
& \leq & \sum_{n=K}^{\infty} \PP[ \log^+ \log^+ ||Z_0||> \log(e^n-n)]\\
& \leq & \sum_{n=K}^{\infty} \PP[ \log^+ \log^+ ||Z_0||> n- \frac{1}{2}].
\end{eqnarray*}
The last series is finite if $\E[\log^+\log^+ ||Z_0||]<\infty$. The Borel-Cantelli-Lemma then shows that a strictly stationary solution exists.

Sharpness:
We denote the $n$-th component of $Y_t$ and $Z_t$ by $Y_t^{(n)}$ and $Z_t^{(n)}$. One can show that a solution has to fulfil $Y_t^{(n)}= \sum_{i=0}^{n}(\prod_{j=i+1}^n a_j) Z_{t+i-n}^{(i)}$. We also know that $\lim_{n \to \infty} Y_t^{(n)}=0$ almost surely because the solution is in $c_0(\N)$. If we assume that only $Z_t^{(0)}$ is nondeterministic and all other component vanish, then the $\log^+\log^+$-moment condition is in fact necessary: then the components of $Y_t^{(n)}$ are  independent (for fixed $t$) and we get by the Borel-Cantelli-Lemma, because only finitely many summands have norm greater than $1$, the following inequality. The finiteness of the first series then gives the $\log^+\log^+$-moment:
\begin{eqnarray*}
&& \sum_{n=1}^{\infty} \PP[\log^+\log^+|Z_{0}^{(0)}|> n]\\
&=& \sum_{n=1}^{\infty} \PP[\log^+|Z_{0}^{(0)}|>\underbrace{-\sum_{i=1}^{n} \log a_i}_{=e^n}]\\
&=& \sum_{n=1}^{\infty} \PP[|Z_{0}^{(0)}|>\prod_{i=1}^{n} \frac{1}{a_i}]\\
&=& \sum_{n=1}^{\infty} \PP[|(\prod_{i=1}^{n} a_i) Z_{t-n}^{(0)}|>1]<\infty.
\end{eqnarray*}
Finally, let $P$ be a real-valued random variable with Pareto distribution on $[1,\infty)$ with index $\alpha=1$. Define $Z=xe^P$ where $x$ is a vector in $\mathcal{B}$ with norm $1$. $Z$ has $\log^+\log^+$-moment but no finite $\log^+$-moment, i.e. there is white noise $(Z_t)_{t \in \Z}$ fulfilling these moment conditions.

\end{proof}

The spectrum of the operator in the next example contains zero but the $\log^+$ condition is necessary.
\begin{example}
Let $A_1$ be the rescaled right shift operator on $\ell^2(\N)$ given by $A_1(x_0,x_1,\ldots) = \frac{1}{2}(0,x_0, x_1,\ldots)$. Then $A_1$ is not invertible, thus we have $0 \in \sigma(A_1)$, but $A_1$ has a left inverse given by $A_1^{-1}(x_0,x_1,\ldots) = 2(x_1,x_2,\ldots)$. It is known that $\sigma(A_1)=\{z \in \C: |z| \leq \frac{1}{2}\}$. The arguments of the proof of Theorem \ref{necessity} show that the $\log^+$-condition is necessary.
\end{example}

We now give another example with a different sufficient condition. Let $\Gamma$ be the gamma function given by $\Gamma(z)=\int_0^{\infty} t^{z-1} e^{-t} \, dt$ for $z > 0$. We know that $\Gamma(n)=(n-1)!$. Let $K > 0$ be a constant such that $\Gamma$ is strictly increasing on $[K,\infty)$. We then denote by $\Gamma^{-1}$ the inverse function of $\Gamma$ on the interval $[K,\infty)$.

\begin{example} Let $\mathcal{B} = C[0,1]$ be endowed with the supremum norm. Let $A_1$ be the Volterra integral operator given by $A_1x(s)= \int_0^s x(t) \, dt$. Then $\E [\Gamma^{-1}(||Z_0|| \vee K)] < \infty$ is a sufficient condition for the existence of a solution for the AR($1$) equation given by $A_1$. This condition is sharp as well.\end{example}
\begin{proof} We will show that $Y_t:=\sum_{n=0}^{\infty} A_1^n Z_{t-n}$ converges almost surely absolutely and we know that this then defines a solution for the AR($1$) equation. We know that $||A_1^n|| = \frac{1}{n!}$, see for example \cite{Conway} p.217 . It suffices to show that $||\frac{1}{n!} Z_{t-n}|| > \frac{1}{n(n-1)}$ for only finitely many $n$. We do so by using the Borel-Cantelli-Lemma:
\begin{eqnarray*}
& & \sum_{n=2}^{\infty} \PP\Big[||\frac{1}{n!} Z_0|| > \frac{1}{n(n-1)}\Big]\\
& = & \sum_{n=2}^{\infty} \PP[ ||Z_0|| > (n-2)!]\\
& \leq  & \sum_{n=2}^{\infty} \PP[ ||Z_0|| \vee K > (n-2)!]\\
& \leq  & \sum_{n=2}^{\infty} \PP[ \Gamma^{-1}(||Z_0|| \vee K) > (n-1)].
\end{eqnarray*}
The last series is finite because of the moment asumption.
Finally, we show that this condition cannot be improved: let $(\widetilde{Z_t})_{t \in \Z}$ be $\C$-valued i.i.d. white noise and define $Z_t=1_{[0,1]}\widetilde{Z_t}$. Now we assume that there is a solution $Y_t$. We evaluate this solution as a function at $1$, i.e. $\widetilde{Y_t}= Y_t(1)= \sum_{n=0}^{\infty} \frac{1}{n!} \widetilde{Z_{t-n}}$. This series necessarily converges if there is a solution $Y_t$. Hence again by the Borel-Cantelli-Lemma, we get $\sum_{n=0}^{\infty} \PP[|\widetilde{Z_{t-n}}|>n!]<\infty$, thus $\sum_{n=0}^{\infty} \PP[|\widetilde{Z_0}| \vee K >n!]<\infty$ and $\sum_{n=0}^{\infty} \PP[\Gamma^{-1}(|\widetilde{Z_0}| \vee K) >n+1]<\infty$. This shows that under this choice of white noise, the moment condition is necessary.

Finally, we give an example of a distribution with $\Gamma^{-1}$-moment but without $\log^+$-moment:
It can be shown that $\Gamma^{-1}(x)$ behaves asymptotically like $\frac{\log x}{\log \log x}$: by Stirling's formula, $\Gamma(y)$ behaves asymptotically  like $x:=\sqrt{y}(\frac{y}{e})^y$. Taking the logarithm, we obtain
\[\log \sqrt{y} + y(\log y  - 1) = \log x =: z.\]
By using the ansatz $y=\frac{z}{\log z}r(z)$ and inserting it in the last equation, one can show that $\lim_{z \to \infty} r(z)=1$, giving the desired asymptotic behaviour of $\Gamma^{-1}(x)$ as $x \to \infty$.

Now, define $f(x):=\frac{1}{x(\log x)^2 (\log \log x)}$. Then, for large $x_1$
\[\int_{x \geq x_1} f(x) \log x \, dx = \int_{x \geq x_1} \frac{1}{x \log x \log\log x} \, dx =\int_{y \geq \log x_1} \frac{1}{e^y y \log y}e^y \, dy = \int_{y \geq \log x_1} \frac{1}{y \log y} \, dy\]
and
\[\int_{x \geq x_1} f(x) \frac{\log x}{\log \log x} \, dx= \int_{y \geq \log x_1} \frac{1}{y (\log y)^2} \, dy .\]
Antiderivatives of the integrands are $\log \log y$ and $\frac{-1}{\log y}$. Hence by restricting $f$ on an interval $[x_1,\infty)$ and normalising it, we obtain a density that defines a distribution with the desired properties of  having finite $\Gamma^{-1}$-moment without having $\log^+$-moment.

\end{proof}

\subsection{Examples when $\sigma(A_1)$ is the closed unit disc}

We consider two further examples. In the first one the operator is an isometry and there is no nondeterministic solution, while in the second one the operator is a multiplication operator and there is a nondeterministic solution.

For the proof of the next proposition, we need some preparation:
a family of random variables $(X_i)_{i \in I}$ is said to \textsl{bounded in probability}, if for each $\varepsilon>0$ there exists a $\delta>0$ such that
\[\PP[||X_i||>\delta] < \varepsilon \mbox{ for all } i \in I.\]
Recall that if the partial sums $S_n$ of independent random variables in a separable space are bounded in probability, then they are also almost surely bounded, see \cite{Vakhania}, Theorem V.2.2. b) p. 266. Note that \cite{Vakhania} define boundedness in probability by boundedness with respect to a metric space that induces convergence in probability. Then they show that it coincides with this definition (Proposition III.1.2. p. 93).
Recall that almost sure boundedness means that the supremum is almost surely finite, i.e. $\PP[\sup_n ||S_n||<\infty]=1$. Finally, recall that a Banach space $\mathcal{B}$  does not contain subspaces being isomorphic to $c_0(\N)$ if and only if for every sequence of independent symmetric Radon random variables in $\mathcal{B}$, the almost sure boundedness of the sequence of the partial sums $S_n$ implies its convergence, see \cite{Ledoux}, Theorem 9.29.

\begin{propo} Let $A_1$ be an isometry and $\mathcal{B}$ be a separable Banach space, that does not contain a subspace being isomorphic to $c_0(\N)$, e.g. a Hilbert space. Then a necessary condition for the existence of a strictly stationary solution for the AR($1$) equation $Y_t-A_1 Y_{t-1}=Z_t$ with $(Z_t)_{t \in \Z}$ i.i.d. is that $Z_0$ is almost surely deterministic.
\end{propo}
\begin{proof}
We will mimic the argument of the proof of Theorem 1 in \cite{BrockwellLindner} c). We may symmetrise the processes by taking the differences of two independent copies of the white noise and the solution. We then show that $Z_0$ vanishes almost surely. We now assume that the white noise and the solution are symmetric.
We again obtain
\[Y_t - A_1^n Y_{t-n} = \sum_{k=0}^{n-1} A_1^k Z_{t-k}.\]
From the stationarity and the fact that $A_1^n$ are isometries, for each $\varepsilon>0$ there is a $K>0$ such that
\[\PP[||Y_t - A_1^n Y_{t-n}|| > K] < \varepsilon\]
as we have
\[ \PP[||Y_t - A_1^n Y_{t-n}|| > K] \leq \PP[\{||Y_t|| > \frac{K}{2}\} \cup \{||A^nY_t|| > \frac{K}{2}\}] \leq 2\PP[||Y_t|| > \frac{K}{2}].\] 
Hence
\[\PP[||\sum_{k=0}^{n-1} A_1^k Z_{t-k}|| > K] < \varepsilon\]
and the $\sum_{k=0}^{n-1} A_1^k Z_{t-k}$ are bounded in probability, hence almost surely bounded and thus $\sum_{k=0}^{\infty} A_1^k Z_{t-k}$ converges almost surely (by the theorems refered to above).
We now have
\[\sum_{k=0}^{\infty} \PP[||Z_0||>r]=\sum_{k=0}^{\infty} \PP[||A_1^kZ_{t-k}||>r] < \infty \quad \mbox{for all } r>0.\]
Thus $Z_0$ vanishes almost surely.
\end{proof}

\begin{example}
Let $\mathcal{B}$ be $\ell^{\infty}(\N)$ and let $A_1$ be the unilateral right shift defined by $A_1(x_0,x_1,\ldots)=(0,x_1,x_2,\ldots)$. We set $Z_t=(Z_t^{(0)},Z_t^{(1)},\ldots)$ and $Y_t=(Y_t^{(0)},Y_t^{(1)},\ldots)$. Then $Y_t^{(n)}:=\sum_{i=0}^n Z_{t+i-n}^{(i)}$ is a stationary solution as long as $Y_t$ is in $\ell^{\infty}(\N)$. If we assume $\sup_{\omega \in \Omega} \sup_{t \in \Z} |Z_t^{(i)}| \leq 2^{-i}$, then the $Y_t^{i}$ are always bounded by $2$, hence a strictly stationary solution exists though the shift is an isometry. This example does not contradict our proposition above as we have $c_0(\N) \subset \ell^{\infty}(\N)$ and $\ell^{\infty}(\N)$ is not separable.
\end{example}

We again consider the AR($1$) equation in the case that $A_1$ is a multiplication operator given by $A_1(x_0,x_1, \ldots) = (\lambda_0 x_0, \lambda_1 x_1, \ldots)$ on $\ell^2(\N)$. We set $Z_t=(Z_t^{(0)},Z_t^{(1)},\ldots)$ and $Y_t=(Y_t^{(0)},Y_t^{(1)},\ldots)$. Define $\sigma_i^2 = \Var(Z_t^{(i)})$. If we consider a component of a solution of the AR($1$) equation $Y_t - A_1 Y_{t-1}=Z_t$, we have $Y_t^{(n)}=\sum_{i=0}^{\infty} \lambda_n^i Z_{t-i}^{(n)}$, hence $\Var(Y_t^{(n)})= \sum_{i=0}^{\infty} (\lambda_n^i)^2 \Var(Z_{t-i}^{(n)})= \frac{\sigma_n^2}{1-\lambda_n^2}$. If $\sum_{n=0}^{\infty} \Var(Y_t^{(n)})$ is finite, then $Y_t$ is almost surely in $\ell^2(\N)$ and we choose a modification of $Y$ that is in $\ell^2(\N)$. Hence we get:

\begin{propo} \label{stronglystablear}Let $A_1$ be a multiplication operator in $\ell^2(\N)$ defined by a sequence $(\lambda_i)_{i \in \N}$ with $|\lambda_i|<1$ for all $i$. Let $(Z_t)_{t \in \Z}$ be i.i.d. and with the notation introduced above assume $\sum_{i=0}^{\infty} \sigma_i^2<\infty$. Then there is a solution for the AR($1$) equation given by $A_1$ if
\[\sum_{i=0}^{\infty} (\frac{\sigma_i^2}{1-\lambda_i^2}) < \infty.\]
If the white noise is Gaussian, then this condition is also necessary. In particular, $\E[||Z_t||^2]<\infty$ is not a sufficient condition for the existence of a strictly stationary solution.
\end{propo}
\begin{proof} The sufficiency has already been observed above. To show necessity when $(Z_t)_{t \in \Z}$ is Gaussian,
consider only the first $n+1$ components of the AR($1$) equation. For this multivariate equation, there is always a solution and the distribution of the solution $(Y_t^{(0)},\ldots,Y_t^{(n)})$ is Gaussian. If $Y_t$ is well-defined in its state space $\ell^2(\N)$, then its distribution is Gaussian, because every linear functional of $Y_t$ is an almost sure limit of the linear functionals of $(Y_t^{(0)},\ldots,Y_t^{(n)})$. These are normally distributed and the normal distributions are closed under convergence in distribution. The diagonal entries of the covariance operator of the distribution of $Y_t$ are given by $\frac{\sigma_i^2}{1-\lambda_i^2}$. Recall that the diagonal entries of a covariance operator of a Gaussian measure on $\ell^2(\N)$ have to be summable, see \cite{Vakhania}, Theorem V.5.6 p. 334.
\end{proof}

We know that the spectrum of the unilateral right shift $R$ in $\ell^2(\N)$ has the closed unit disc as spectrum. We also know that the spectrum of the multiplication operator $M$ is $\overline{\{\lambda_i: i \in \N\}}$. We can now choose a sequence $(\lambda_i)_{i \in \N}$ with $|\lambda_i|<1$, such that $\sigma(R)=\sigma(M)$. The shift is an isometry, hence we already know that the corresponding AR($1$) equation has no strictly stationary solution but we also know that in the case of the multiplication operator $M$ we have a solution, if we choose a priori appropiate white noise.

The operator in the next example has the disc with radius 2 as spectrum. 
\begin{example}
Let $A_1$ be the rescaled right shift operator on $\ell^2(\N)$ given by $A_1(x_0,x_1,\ldots) = 2 (0,x_0, x_1,\ldots)$. Then $A_1$ has a left inverse given by $A_1^{-1}(x_0,x_1,\ldots) = \frac{1}{2}(x_1,x_2,\ldots)$. There is no solution for the corresponding AR($1$) equation if $Z_0$ is nondeterministic.
\end{example}
\begin{proof}
We mimic the arguments given in \cite{BrockwellDavis} p. 81 for the noncausal univariate AR($1$) equation. If there is a solution for the AR equation, we multiply the equation with $A_1^{-1}$ and get $A_1^{-1}X_t - X_{t-1} = A_1^{-1} Z_t$ and rearrange it to $X_t=-A_1^{-1} Z_{t+1} + A_1^{-1} X_{t+1}$. Iteration gives $X_t=-A_1^{-1} Z_{t+1} - \cdots -A_1^{-n} Z_{t+n} + A_1^{-n} X_{t+n}$. By taking the limit in distribution and applying again the Itô-Nisio-Theorem one sees that  if there is a solution for the corresponding AR($1$) equation for $A_1$, then it is unique and fulfils $X_t=-\sum_{j=1}^{\infty} A_1^{-j} Z_{t+j}$ as an almost sure limit. Hence on the one hand, we have $X_t^{(0)}=-\sum_{j=1}^{\infty} 2^{-j} Z_{t+j}^{(j)}$. On the other hand we also have $X_t^{(0)}=Z_t^{(0)}$ but this is impossible unless $Z_0$ is deterministic.
\end{proof}

\section{Summary}
The aim of this article was to find necessary and sufficient conditions for the existence of strictly stationary solutions of ARMA equations in Banach spaces with i.i.d. white noise.

Firstly, we generalised a result on multivariate ARMA($1$,$q$) processes by Brockwell, Lindner and Vollenbröker \cite{BrockLindVoll} by excluding the unit circle and  zero from the spectrum of the operator in the autogressive part. Secondly, we extended this to ARMA($p$,$q$) processes by using their representations as ARMA($1$,$q$) processes.  Thirdly, we gave an additional representation of the solution by using Laurent series. Finally, we provided various examples illustrating what can happen when we drop our assumptions. These examples show that a complete characterisation of the existence and uniqueness of strictly stationary solutions by the spectrum and moment conditions is not possible.

Our work extends the results in the book by Bosq \cite{Bosq} in that we allow for noncausal solutions, a broader class of operators, white noise without finite second moments and a moving average part.


\begin{thebibliography}{99}
\bibitem{Besse} P. C. Besse, H. Cardot, D.B. Stephenson Autoregressive forecasting of some functional climatic variations \textsl{Scand. J. Statist.}  \textbf{27}:673-687, 2000
\bibitem{Bosq} D. Bosq \textsl{Linear Processes in Function Spaces} Springer-Verlag, New York 2000
\bibitem{BrockwellDavis} P. J. Brockwell, R. A. Davis \textsl{Time Series: Theory and Methods} Second Edition Springer-Verlag, New York 2006
\bibitem{BrockwellLindner2} P. J. Brockwell, A. Lindner Existence and uniqueness of stationary Lévy driven CARMA processes \textsl{Stoch. Proc. Appl.}  \textbf{119}, 2625-2644, 2009
\bibitem{BrockwellLindner} P. J. Brockwell, A. Lindner Strictly Stationary Solutions of Autoregressive Moving Average Equations \textsl{Biometrika}  \textbf{97}, 765-772, 2010
\bibitem{BrockLindVoll} P. J. Brockwell, A. Lindner, B. Vollenbr\"oker  Strictly stationary solutions of multivariate ARMA equations with i.i.d. noise \textsl{Ann. Inst. Statist. Math.}, to appear
\bibitem{Conway2} J. B. Conway \textsl{Functions of One Complex Variable} Springer-Verlag, New York 1975
\bibitem{Conway} J. B. Conway \textsl{A Course in Functional Analysis} Springer-Verlag, New York 1985
\bibitem{Eisner} T. Eisner \textsl{Stability of Operators and Operator Semigroups} Birkhäuser, Basel 2010
\bibitem{Engel} K.-J. Engel, R. Nagel \textsl{A Short Course on Operator Semigroups}, Springer-Verlag, New York 2006
\bibitem{Hoermann} S. Hörmann, P. Kokoszka Functional Time Series, in  \textsl{Handbooks of Statistics}, Eds. Subbu Rao, Subbu Rao, Rao,  \textbf{30}, Elsevier, Amsterdam 2012
\bibitem{Kato} T. Kato \textsl{Pertubation Theory  for Linear Operators} Springer-Verlag, New York 1966
\bibitem{Ledoux} M. Ledoux, M. Talagrand \textsl{Probability in Banach Spaces} First Reprint, Springer-Verlag, New York 2002
\bibitem{Segers} T. Meinguet, J. Segers Regularly varying time series in Banach spaces \textsl{Preprint} 2010
\bibitem{Mas} A. Mas, B. Pumo Linear processes for functional data, in \textsl{The Oxford Handbook of Functional Data}, Ferraty and Romain Eds., Oxford 2010 
\bibitem{Mikosch} T. Mikosch, D. Straumann Quasi-MLE in heteroscedastic time series: a stochastic recurrence equations approach \textsl{Annals of Statistics}  \textbf{34}, 2449-2495,  2006
\bibitem{Klueppelberg} R. Sen, C. Klüppelberg Time series of functional data \textsl{Preprint} 2010
\bibitem{Vakhania} N. N. Vakhania, V. I. Tarieladze, S. A. Chobanyan \textsl{Probability Distributions on Banach Spaces} D. Reidel Publishing Company, Dordrecht 1987
\end{thebibliography}
\end{document}